\documentclass[12pt,reqno]{amsart}
\usepackage{amsfonts,amssymb,amsmath,amsopn,amsthm,graphicx}
\usepackage{amsxtra, mathrsfs}
\usepackage{colordvi}
\usepackage[usenames,dvipsnames]{color}
\usepackage{amsfonts,amssymb,amsbsy,amsmath,amsthm,dsfont}

\usepackage{amsmath,amssymb,amsfonts,amscd,color}
\usepackage[latin1]{inputenc}
\usepackage{verbatim}

 \numberwithin{equation}{section}
\newcommand{\Hmm}[1]{\leavevmode{\marginpar{\tiny%
			$\hbox to 0mm{\hspace*{-0.5mm}$\leftarrow$\hss}%
			\vcenter{\vrule depth 0.1mm height 0.1mm width \the\marginparwidth}%
			\hbox to
			0mm{\hss$\rightarrow$\hspace*{-0.5mm}}$\\\relax\raggedright #1}}}

\newcommand{\bb}{{\mathbf{\overline {b}}}}
\newcommand{\bt}{{\mathbf{\tilde{b}}}}

\newcommand{\loc}{{\rm loc}}

\newcommand{\N}{\mathbb{N}}

\newcommand{\R}{\mathbb{R}}

\newtheorem{theorem}{Theorem}[section]

\newtheorem{thm}[theorem]{Theorem}
\newtheorem{lem}[theorem]{Lemma}
\newtheorem{lemma}[theorem]{Lemma}
\newtheorem{proposition}[theorem]{Proposition}
\newtheorem{example}[theorem]{Example}

\newtheorem{definition}[theorem]{Definition}
\newtheorem{defi}[theorem]{Definition}
\newtheorem{remark}[theorem]{Remark}
\newtheorem{rem}[theorem]{Remark}

\theoremstyle{definition}

\newcommand{\diver}{\mathrm{div}\,}
\newcommand{\dx}{\,\mathrm{d}x}
\newcommand{\dy}{\,\mathrm{d}y}
\newcommand{\dz}{\,\mathrm{d}z}
\newcommand{\dt}{\,\mathrm{d}t}

\newcommand{\ds}{\,\mathrm{d}s}

\newcommand{\dtau}{\,\mathrm{d}\tau}

\newcommand{\dsigma}{\,\mathrm{d}\sigma}

\newcommand{\core}{C_0^{\infty}(\Omega)}

\newcommand{\be}{\begin{equation}}
\newcommand{\ee}{\end{equation}}
\newcommand{\bea}{\begin{eqnarray}}
\newcommand{\eea}{\end{eqnarray}}
\newcommand{\bean}{\begin{eqnarray*}}
	\newcommand{\eean}{\end{eqnarray*}}

\newcommand{\opname}[1]{\mbox{\rm #1}\,}
\newcommand{\supp}{\opname{supp}}

\newlength{\wex}  \newlength{\hex}
\def\ga{\alpha}     \def\gb{\beta}       
             
                         \def\vge{\varepsilon}
\def\gf{\phi}       \def\vgf{\varphi}    
            \def\gl{\lambda}

\def\Gg{\Gamma}     \def\Gd{\Delta}

\def\Gw{\Omega}              

\begin{document}
	
	\title[Families of optimal Hardy-weights]{On families of optimal Hardy-weights for linear second-order elliptic operators}
	%
	\author{Yehuda Pinchover}
	
	\address{Yehuda Pinchover,
		Department of Mathematics, Technion - Israel Institute of
		Technology,   Haifa, Israel}
	
	\email{pincho@technion.ac.il}
	
	\author {Idan Versano}
	
	\address {Idan Versano, Department of Mathematics, Technion - Israel Institute of
		Technology,   Haifa, Israel}
	
	\email {idanv@campus.technion.ac.il}
	\begin{abstract}
We construct families of optimal Hardy-weights for a subcritical linear second-order elliptic operator using a one-dimensional reduction. More precisely, we first  characterize all optimal Hardy-weights with respect to one-dimensional subcritical Sturm-Liouville operators on $(a,b)$, $\infty\leq a<b\leq \infty$, and then apply this result to obtain families of optimal Hardy inequalities for general linear second-order elliptic operators in higher dimensions. As an application,  we prove a new Rellich inequality.  
	
		\medskip
		
		\noindent  2000  \! {\em Mathematics  Subject  Classification.}
		Primary  \! 35B09; Secondary  35J08, 35J20, 47J20, 49J40.\\[1mm]
		\noindent {\em Keywords:} ground state, Hardy inequality, minimal growth, positive solutions, Rellich inequality, variational inequalities.
	\end{abstract}
	
	\maketitle

	%
	\section{Introduction}
The classical Hardy inequalities in the continuum and discrete cases were introduced in 1920th, and has been evolved to many versions in different aspects of mathematics, see \cite{BEL,D,KO,KML} for historical reviews. The problem of improving Hardy-type inequalities has engaged many authors till today, see \cite{BEL,BFT,BM}, and reference therein.  Fraas, Devyver and Pinchover, established  a general method to generate an explicit {\em optimal} Hardy-weight for a general (not necessarily symmetric) subcritical linear elliptic operators either in divergence form or in non-divergence form \cite{DFP}. In some definite sense, an optimal Hardy-weight is {\em `as large as possible'} Hardy-weight. Such an optimal weight is by no means unique. In the present paper we construct for a given linear second-order nonnegative elliptic operator $P$ {\em families} of optimal Hardy-weights, generalizing the results in \cite{DFP}. In fact, for the one-dimensional case, we characterize the set of all optimal Hardy-weights. We then apply this characterization to obtain families of optimal Hardy inequalities for general linear second-order elliptic operators in higher dimensions.
	 
	 \medskip 
	 
Let us first recall the notion of optimal Hardy-weights and present the main result of \cite{DFP}.	
	\begin{definition}
		{\em Let $\Omega$ be a domain in $\mathbb{R}^n$ (or a noncompact $n$-dimensional Riemannian manifold), and consider a second-order linear elliptic operator $P$ defined on $\Gw$. 
			\begin{itemize}
				\item We say that $P$ is {\em nonnegative} in $\Gw$ (in short, $P\geq 0$ in $\Gw$) if $P$ admits a positive global  (super)solution of the equation $Pu=0$ in $\Gw$. 
				\item  We say that a nonzero nonnegative function $W$ is a {\em Hardy-weight} for $P$ in $\Gw$ if the following {\em Hardy-type inequality}
				$P-W\geq 0$ in $\Gw$  holds true.  
		\item	A nonnegative operator $P$ in $\Gw$ is said to be {\em subcritical}  (respect., {\em critical})  in $\Gw$ if $P$ admits (respect., does not admit) a Hardy-weight for $P$ in $\Gw$.   
	\end{itemize}
}
\end{definition}

It is well-known that $P$ is subcritical in $\Gw$ if and only if  it admits a positive minimal Green function $G_{P}^{\Gw}(x,y)$ in $\Gw$. On the other hand,  $P$ is critical in $\Gw$ if and only if the equation $Pu=0$ admits (up to a multiplicative constant) a unique positive supersolution $\phi$ in $\Gw$. In fact, $\phi$  is  a positive solution, called the {\em Agmon ground state}. Clearly, $P$ is critical in $\Gw$ if and only if $P^\star$ is critical in $\Gw$, where $P^\star$ is the formal adjoint of $P$ in $L^2(\Gw)$ \cite{P3}.

	\begin{definition}
		{\em Let $P$ be nonnegative in $\Gw$. A Hardy-weight $W$ is said to be optimal in $\Gw$ if $P-W$ is critical in $\Gw$, and $\phi \phi^*\not\in L^1(E, W\dx)$, where $\phi$ and $\phi^*$ are respectively, the ground states of $P-W$ and $P^\star-W$ in $\Gw$, and $E\subset \Gw$ is any end of $\Gw$.   In this case we say that $P-W$ is {\em null-critical at each end of $\Gw$} with respect to the weight $W$.			
}
\end{definition}
\begin{rem}\label{rem_1.3}{\em 
	We say that a Hardy-weight $W$ is {\em optimal at infinity in $\Gw$} if for any $K\Subset \Gw$, 
	$$\sup\{\gl \in \R \mid P-\gl W\geq 0 \mbox{ in } \Gw\setminus K\}=1.$$   	
We note that the original definition of optimal Hardy-weights in \cite{DFP} includes the requirement that $W$ is optimal at infinity. Recently it was proved in \cite[Corollary 3.4]{KP} that the latter property  follows, in fact,  from the null-criticality of $P-W$ with  respect to $W$.	 
}
\end{rem}
The main result of Fraas, Devyver and Pinchover (\cite[Theorem 4.12]{DFP}, see also \cite[Theorem 2.15]{KP}) reads as follows.	
		\begin{thm}\label{thm:DFP} 
		Let $P$ be a subcritical operator in $\Gw$, and let $G_{P}^{\Gw}$ be its positive minimal Green function. Let $0\lneqq \varphi \in \core$,
		and consider the Green potential 
		\begin{equation}\label{green potential}
		G_{\varphi}(x):= \int_{\Gw} G_{P}^{\Gw}(x,y) \varphi(y) \dy,
		\end{equation}
		(so, $PG_{\varphi}=\varphi$ in $\Gw$).
		Suppose that there exists a positive solution $u$ of the equation $P=0$ in $\Gw$ satisfying 
		\begin{equation}\label{Ancona}
		\lim_{x\to \overline{\infty}} \dfrac{G(x)}{u(x)}=0,
		\end{equation}
		where $\overline { \infty }$ denotes the ideal point in the one-point compactification of $\Omega$.
		
		Consider the supersolution $$v:=\sqrt{G_{\varphi}u}$$  of the operator $P$ in $\Gw$. Then the associated 
		Hardy-weight $$W:= \dfrac{Pv}{v} $$ is an optimal Hardy-weight 
		with respect to $P$ in $\Gw.$ Moreover, 
		$$
		W=\dfrac{1}{4} \left | \nabla \left ( \log \dfrac{G_{\varphi}}{u}\right ) \right |_{A}^2 = \dfrac{1}{4} \dfrac{\left | \nabla \left (  (G_{\varphi}/u)\right ) \right |_{A}^2}{(G_{\varphi}/u)^2} \qquad
		 \text{in} ~ \Gw \setminus \supp \varphi,
		$$
		where $A$ is the diffusion matrix associated with the principal part of the operator $P$, $| \xi|_A:=\sqrt{\langle A\xi,\xi \rangle }$, and $\langle \cdot, \cdot \rangle$ denotes the  Euclidean inner product on $\R^n$. 	
	\end{thm}
\begin{rem}{\em
		 Since $G_{\varphi} \asymp G$ in a neighborhood of $\bar\infty$, the theorem's assumption  
$\lim_{x\to \overline{\infty}} \big(G(x)/u(x)\big)=0$ clearly implies $\lim_{x\to \overline{\infty}} \big(G_{\varphi}(x))/u(x)\big)=0$. We recall that if $P$ is symmetric, or more generally quasi-symmetric, then a global positive solution $u$ satisfying \eqref{Ancona} always exists \cite{Ancona02}.
	}
\end{rem}
\begin{rem}{\em 
The above result has been extended to the case of the $p$-Laplace operator by Devyver and Pinchover \cite{DP1} and recently to the realm of Schr\"odinger operators on discrete graphs by  Keller, Pinchover, and Pogorzelski \cite{KPP}.
	}
\end{rem}
	The optimal Hardy-weights obtained by Theorem~\ref{thm:DFP}  are obviously not unique even in the one-dimensional case. The main aim of the present paper is to {\em characterize} the set of all {\em optimal} Hardy-weights for Sturm-Liouville operators, and to establish a new construction of families of optimal Hardy-weights for general second-order subcritical linear elliptic operators using a one-dimensional reduction.         		

\medskip

The paper is organized as follows. In Section~\ref{sec_prelim}, we introduce the necessary notation and recall some previously obtained results needed in the present paper. We proceed in Section~\ref{sec-1d},	 with the characterization of the set of all optimal Hardy weights for Sturm-Liouville operators, while in Section~\ref{sec_EP_eq}, we use the explicit  form of the general  solution of the (semilinear) Ermakov-Pinney equation to obtain a family of strictly positive Hardy-weights for one-dimensional subcritical Schr\"odinger operators which are optimal at one end. Applying the above characterization of optimal Hardy-weights for the one-dimensional case, we obtain in Section~\ref{sec-optimal_higher_d}  families of Hardy inequalities in the higher dimensional case. More precisely, in the symmetric case, our construction leads to a family of {\em optimal} Hardy weights, while in the nonsymmetric, case we are only able to prove that the aforementioned one-dimensional reduction gives rise  to critical Hardy-weights that are optimal at infinity. Finally, in Section~\ref{sec-improved-optimal}, we use the Ermakov-Pinney equation to obtain in the nonsymmetric higher-dimensional case, a new family of optimal Hardy-weights which are larger at infinity than the optimal Hardy-weights obtained in \cite{DFP}.

\medskip

Throughout the paper we use the following notation.  Let $\Gw$ be a domain in  $\R^n$, $ n\geq 1$.
\begin{itemize}
	\item For any $R>0$ and $x\in \R^n$, we denote by $B_R(x)$ the open ball of radius $R$ centered at $x$.
	\item $\mathrm{AC}_{\mathrm{loc}}(I)$ refers to the space of all absolutely continuous real functions defined on an interval $I\subset \R$.  
	\item For any $\xi\in \R^n$ and a positive definite matrix $A\in \R^{n \times n}$, let  
	$| \xi|_A:=\sqrt{\langle A\xi,\xi \rangle }$, where $\langle \cdot, \cdot \rangle$ denotes the Euclidean inner product on $\R^n$.
	\item We write $\Gw_1 \Subset \Gw_2$ if $\Gw_2$ is open in $\Gw,$ the set $\overline{\Gw_1}$ is compact, and $\overline{\Gw_1}\subset \Gw_2$. 
	\item $C$ refers to a positive constant which may vary from  line to line.
	\item Let $g_1,g_2$ be two positive functions defined in $\Gw$. We use the notation $g_1\asymp g_2$ in
	$\Gw$ if there exists a positive constant $C$ such
	that
	$$C^{-1}g_{2}(x)\leq g_{1}(x) \leq Cg_{2}(x) \qquad \mbox{ for all } x\in \Gw.$$	
\end{itemize}
	\section{Preliminaries}\label{sec_prelim}
	Let $\Omega$ be a domain in $\mathbb{R}^n$.
	We consider a second-order linear elliptic operator $P$ with real coefficients defined in $\Gw$ which is either
	in divergence form
	\begin{equation}\label{eq:div_form_linear}
	Pu:=-\diver \left [ A(x)\nabla u+u \bt (x)\right ]+\bb(x)\cdot \nabla u+c(x)u,
	\end{equation}
	or in the form
	\begin{equation}\label{eq:non_div_form_linear}
	Pu:=-\sum_{i,j=1}^{n}a^{ij}(x)\partial_i \partial_ju+\mathbf{b}(x)\cdot \nabla u+c(x) u,
	\end{equation}
	where $A(x)=(a_{ij})\in \R^{n \times n}$ is a symmetric positive definite matrix in $\Gw$, and  $\bt,\bb,\mathbf{b},$ are vector fields.
	We also assume that  $P$ is locally uniformly elliptic, that is, for any $K\subset \subset \Omega$ there exist $0<\lambda_K\leq \Lambda_K$ such that for all
	$\xi \in \mathbb{R}^n$ and $x\in K$
	$$
	\lambda _K|\xi|^2 \leq | \xi|_A^2=\sum_{i,j=1}^{n}a^{ij}(x)\xi_i\xi _j \leq \Lambda_K |\xi|^2.
	$$
	We further assume the following regularity assumptions on the coefficients of $P$:
	\begin{itemize}
	\item If $P$ is of the form \eqref{eq:div_form_linear}, then we further assume that $A$ is measurable, $\bb,\bt \in L^{p}_\loc(\Gw,\R^n)$ and $c\in L^{p/2}_\loc(\Gw)$ for some $p>n$.
	 \item 
	 If $P$ is of the form \eqref{eq:non_div_form_linear}, then we assume that $A, \mathbf{b},c$ are locally H\"older continuous in $\Gw$.
	 \end{itemize}
 For a divergence form operator $P$, we say that $u$ is a {\em (weak) solution} of $Pu=0$ in $\Gw$ if $u\in W^{1,2}_{\mathrm{loc}}(\Omega)$ and the associated bilinear form $B$ satisfies 
 $$
 B(u,\varphi):=\int_{\Omega} \left(\langle
 A(x) \nabla u,\nabla \varphi\rangle +u\langle\bt, \nabla\varphi\rangle + \langle\bb, \nabla u\rangle\varphi  +cu \varphi \right)\!\!\dx =0
 $$
 for every test function $\varphi\in C^{\infty}_{0}(\Omega)$. 	Similarly, $u\in W^{1,2}_{\mathrm{loc}}(\Omega)$ is a {\em supersolution} (respect., {\em subsolution}) of $Pu=0$ in $\Gw$ if $B(u,\varphi) \geq 0$ (respect., $B(u,\varphi)\leq 0$) for all $0\leq \varphi\in C^{\infty}_{0}(\Omega)$. 	
 We denote by $Q$ the associated quadratic form $Q(\vgf):=B(\vgf,\vgf)$ on $\core$.

By elliptic regularity theory,  each (weak) solution of $Pu=0$ in $\Gw$ is locally H\"older continuous in $\Omega.$
We say that $P$ of the form \eqref{eq:div_form_linear} is {\em symmetric} if $\bt=\bb$ (since in this case the bilinear form $B$ is symmetric on $C^{\infty}_{0}(\Omega)$).

 For an operator $P$ of the form \eqref{eq:non_div_form_linear}, we say that 
$u$ is a {\em (classical) solution} (respect., {\em supersolution}, {\em subsolution}) of $Pu=0$ in $\Gw$ if $u\in C^2(\Gw) $ and
 $$
-\sum_{i,j=1}^{n}a^{ij}(x)\partial_i \partial_ju(x)+\mathbf{b}(x)\cdot \nabla u(x)+c(x) u(x)=0 \;(\mbox{respect.}, \;\geq 0,\;\leq 0 ) \qquad \forall x\in \Gw.
$$
		
	We denote the cone of all positive solutions to the equation $Pu=0$ in $\Gw$ by $\mathcal{C}_{P}(\Omega)$, and by $\mathcal{S}_{P}(\Omega)$ the cone of all positive supersolutions in $\Gw$.
	The operator $P$ is said to be {\em nonnegative} in $\Omega$ if $\mathcal{C}_{P}(\Omega)\ne \emptyset$, (or equivalently, $\mathcal{S}_{P}(\Omega)\ne \emptyset$), and in such a case we write $P\geq 0$ in $\Gw$.

	The  following well known Allegretto-Piepenbrink theorem  (in short, the AP theorem) connects between the nonnegativity of a {\em symmetric} operator $P$ and the nonnegativity of the associated quadratic form $Q$   (see for example \cite{AG,CFKS,P3} and references therein):
	\begin{theorem}[The AP theorem]
			Let $P$ be a symmetric operator. Then the following assertions are equivalent
			\begin{enumerate}
				\item $Q(\varphi) \geq 0$ for all $\varphi \in C_{0}^{\infty}(\Omega)$.
				\item $\mathcal{C}_P(\Gw) \neq \emptyset$.
				\item $\mathcal{S}_P(\Gw) \neq \emptyset$.
			\end{enumerate}
		
	\end{theorem}
	\begin{defi}
		\em{
			A sequence $\{\Omega_k\}_{k\in \mathbb{N}}$ 
			is called an ${exhaustion}$  of $\Gw$ if it is
			an increasing sequence of precompact subdomains of $\Gw$ with smooth boundary such that $\Omega_k\subset\subset \Omega_{k+1}$
			and $\bigcup_k \Omega_k=\Omega$. For $k\geq 1$ we denote $\Omega_{k}^{*}:=\Omega \setminus \overline{\Omega_k}$.
		}
	\end{defi}
	\begin{defi}{\em
			Let $\{\Omega_k\}_{k\in \N}$ be an exhaustion of $\Gw,$ and
			let $K$ be a compact set in $\Omega_{k_0}$. A positive solution $u$ of the equation $Pu=0$ in $\Omega\setminus K$ is said to be {\em of minimal growth
				in a neighborhood of infinity in  $\Gw$}
			if for any $k\geq k_0$ and $v\in  \mathcal{S}_P(\Omega_{k}^{*})\cap C(\overline{\Omega_{k}^{*}})$, the inequality $u\leq v$ on
			$\partial \Omega_k$ implies that $u\leq v$ in $\Omega_{k}^{*}$.}
	\end{defi}
	\begin{rem}
	\em{
		If $u\in \mathcal{C}_P(\Gw\setminus K)$ has minimal growth in  a neighborhood of infinity in $\Gw$,  then for any $v\in \mathcal{S}_P(\Gw \setminus K)$ there exists $C>0$ such that $u \leq C v$ in $\Gw \setminus K_1$, where $K \Subset K_1\Subset \Gw$.
	}
\end{rem}
	\begin{lemma}\cite[Theorem 4.2]{AG}
	Assume that $P\geq 0$ in $\Omega$. Then for any $x_0\in \Omega$ there exists a unique (up to a multiplicative constant) positive solution  $u_{x_0}$ of the equation $Pu=0$
	in $\Omega\setminus\{x_0\}$ of  minimal growth in a neighborhood of infinity in $\Omega$.
\end{lemma}
	\begin{defi}{\em
			Assume that $P\geq 0$ in $\Omega$.
			\begin{itemize}
				\item If there exists $\phi \in \mathcal{C}_P(\Gw)$ of minimal growth in a neighborhood of infinity in $\Gw$, then $\gf$ is said to be a {\em (Agmon) ground state} of $P$ in $\Gw$.
				\item 
				 If the unique (up to a multiplicative constant) positive solution $u_{x_0}$ in $\Omega\setminus\{x_0\}$ of  minimal growth in a neighborhood of infinity in $\Omega$ has a irremovable singularity at $x_0$, then $G_P^\Gw(x,x_0):=\ga u_{x_0}(x)$, where  	
				  $\alpha$ is chosen such that $PG_{P}^{\Gw}(x,x_0)$ equals to the Dirac measure at $x_0$, is called the {\em positive minimal Green function} of $P$ in $\Gw$ with singularity at $x_0$ \cite{Ancona01,P3}.
			\end{itemize}		
		}
	\end{defi}
	\begin{theorem}[{\cite{AG} and \cite[Section 2.1]{DFP}}]\label{thm_crit}
		Assume that $P\geq 0$ in $\Gw$. The following assertions are equivalent:
			
			(a) $P$ is critical in $\Gw$ (meaning that $P$ does not admit a Hardy-weight).
			
			(b) $P$ admits a unique (up to a multiplicative constant) positive supersolution in $\Gw$.
			
			(c) $P$ admits a ground state.
			
			Moreover,  $P$ is subcritical in any subdomain $\Gw_1 \subsetneqq \Gw$. 
	\end{theorem}
	\begin{remark}
		\em{
	Criticality theory has been extended to the case of half-linear operators of Schr\"odinger-type
	\begin{equation}\label{QPV}
	Q_{A,p,V}'(u):=-\mathrm{div}(|\nabla u|_A^{p-2}\nabla u) +V(x)|u|^{p-2}u,
	\end{equation}
	where $1<p<\infty$ \cite{PP,PT}, and to the case of Schr\"odinger operators on discrete graphs \cite{KPP1}.
}
	\end{remark}	
\begin{definition}{\em 
Consider an elliptic operator $P$ on $\Gw$ either of the form \eqref{eq:div_form_linear} or \eqref{eq:non_div_form_linear}. The {\em generalized principal eigenvalue} of $P$ in $\Gw$ with respect to a potential $W\geq 0$ is defined by
	\begin{equation*}
	\lambda_0(P,W,\Omega):=\sup\{\lambda \in \mathbb{R}\mid  P-\lambda W \geq 0 \mbox{ in } \Gw\}.
	\end{equation*}
	We also denote
	\begin{equation*}
	\lambda_{\infty}(P,W,\Omega):=\sup\{\lambda \in \mathbb{R}\mid   \exists K\Subset \Gw ~ \text{ s.t. } P-\lambda W \geq 0 \mbox{ in } \Gw \setminus K\}.
	\end{equation*}
	}
\end{definition}
	We might write $\lambda_0:=\lambda_0(P,W,\Omega)$ and 
	$\lambda_{\infty}:=\lambda_{\infty}(P,W,\Omega)$ when there is no danger of ambiguity. 
	
	Note that by the AP theorem,  if $P$ is symmetric, then $\lambda_0(P,W,\Omega)$ is the best constant $\gl$ satisfying the functional inequality
	$$Q(\varphi)\geq \gl\int_{\Omega} W\varphi^2\dx\qquad \forall \varphi \in \core.$$
	Moreover, assume further that $W$ is strictly positive, then $\gl_0$ and $\gl_\infty$ are the bottom of the $L^2(\Gw,W\!\dx)$-spectrum and essential spectrum of $P$ (see for example \cite{Agmon}).
	 
	 \begin{defi}{\em 
	 		Let $D\Subset \Gw$ be a subdomain of $\Gw\subset \R^n.$ We say that the \textit{generalized maximum principle} for the operator $P$ holds in $D$ if  for any  $u\in  C(\bar D)$ satisfying $Pu\geq 0$ in $D$ and $u \geq 0$ on $\partial D$, we have $u \geq 0$ on $D$.
	 	}
	 \end{defi}
	 
	 It is well known \cite{AG} that $\lambda_0(P,\Gw,1)\geq 0$ if and only if the generalized maximum principle holds true in any  subdomain $D \Subset \Gw$.
\section{Optimal Hardy-inequalities for Sturm-Liouville operators}  \label{sec-1d}	 
In the one-dimensional case, optimal Hardy-weights for a general subcritical Sturm-Liouville operator on an interval $(a,b)$ are characterized  by the following integral conditions:
\begin{proposition}\label{prop_2.10}
	Let $\infty\leq a<b\leq \infty$.  Consider the Sturm-Liouville operator 
	$$
	L(y):=-(py')' + qy \qquad \mbox{on } (a,b),
	$$
	where $q\in L^1_{\mathrm{loc}}(a,b)$, $p\in C^{1,\alpha}(a,b)$, $p>0$.
	
	Then $0\lneqq w\in L^1_{\mathrm{loc}}(a,b) $ is an optimal Hardy-weight for $L$ in $(a,b)$  if and only if  
	there exists a positive function $f_w$ satisfying $(L-w)f_w=0$ in $(a,b)$ such that 
	\begin{equation}\label{eq:2.5}
	\int\limits_a^{c} \frac{1}{pf_w^2} \dt=
	\int\limits_{c}^{b} \frac{1}{pf_w^2} \dt= 
	\int\limits_a^{c} wf_w^2 \dt=
	\int\limits_{c}^{b} wf_w^2 \dt=
	\infty,
	\end{equation}
	for any $a<c < b$. 
\end{proposition}
\begin{proof}
	Assume that $w$ is an optimal Hardy-weight of $L$ and fix
	$a<c < b$.  
	Then $L-w$ is critical and the equation $(L-w)y=0$ in $(a,b)$ admits  a unique positive solution (up to a multiplicative constant) $f_w$.	
	Moreover,
	\begin{equation*}
	\int_{a}^{c}\dfrac{1}{pf_w^2}\ds=
	\int_{c}^{b}\dfrac{1}{pf_w^2} \ds=\infty,
	\end{equation*}
	since otherwise, one of the functions
	\begin{equation*}
	\quad f_w(t)\int\limits_{a}^{t}\dfrac{1}{pf_w^2}\ds, \quad f_w(t)\int\limits_{t}^{b}\dfrac{1}{pf_w^2}\ds,
	\end{equation*}
	would be another linearly independent positive solution of the equation $(L-w)y=0$ in $(a,b)$.
	In addition, the null-criticality of $L-w$ with respect to $w$ just means that for any $a<c<b$ we have
	$$
	\int\limits_a^{c} wf_w^2 \dt=
	\int\limits_{c}^{b} wf_w^2 \dt=
	\infty.
	$$

	\medskip
	
	On the other hand, let $f_w$ be a positive solution of the equation $(L-w)y=0$ in $(a,b)$ satisfying \eqref{eq:2.5}. The general solutions of the equation $(L-w)y=0$ in $(a,b)$ are of the form 
	$$\alpha f_w+\beta f_w \int_{c}^t \dfrac{1}{pf_w^2} \ds, \quad \text{for some} ~ \alpha,\beta \in \R  ~ \text{ and } ~ a<c < b\, .
	$$
	But by \eqref{eq:2.5}, such a solution with $\gb\neq 0$ is negative either near $a$ or near $b$.  
	Hence, $f_w$ is the unique (up to a constant) positive solution of the equation  $(L-w)y=0$ in $(a,b)$. But in the one-dimensional case this implies that $L-w$ is critical in $(a,b)$ \cite[Appendix 1]{MU}.
	Moreover, by definition, \eqref{eq:2.5} implies the null-criticality of $L-w$ with respect to $w$. 
\end{proof}
	 Recall that by Remark~\ref{rem_1.3}, if $W$ is an optimal Hardy-weight for $P$ in $\Gw$, then  $W$ is also optimal at infinity in $\Gw$. In other words, in this case we have $\lambda_{\infty}(P,W,\Omega)=1$. For the one-dimensional case we give a simple alternative proof for this assertion.
	 \begin{proposition} \label{2.14}
	 	Consider a nonnegative  Sturm-Liouville operator $$L(y):=-(py')' + qy \qquad \mbox{on } (a,b),$$ 
	 	where $q\in L^1_{\mathrm{loc}}(a,b)$, $p\in C^{1,\alpha}(a,b)$, $p>0$ in $(a,b)$. If  $w\in L^1_{\mathrm{loc}}(a,b)$ is an optimal Hardy-weight for $L$ in $(a,b)$, then $\lambda_\infty(L,w,(a,b))=1$.
	 \end{proposition}
	 \begin{proof}
	 	Denote by  $f_w$ the ground state of $L-w$ in $(a,b)$. The criticality of $L-w$ in $(a,b)$ implies that for a fixed $a<c<b$ we have 
	 	$$\int_a^c \dfrac{1}{pf_w^2} \dt=\int_c^b \dfrac{1}{pf_w^2} \dt=\infty.$$
	 	Assume by contradiction that there exists $\lambda>0$ such that
	 	$M:=L-(1+\gl)w \geq 0$ in $(a,b) \setminus K$ for some $K\Subset (a,b)$. Note that $-Mf_w=\gl wf_w$. Therefore, 
	 	the null-criticality with respect to $w$ of $L-w$ in $(a,b)$ implies that 
	 	$$\min\Big\{\!\int_a^c\!\!f_w(-Mf_w) \dt,\int_c^{b} \!\!f_w (-Mf_w) \dt\!\Big\}= \lambda\min\Big\{\!\!\int_a^c\!\!wf_w^2\dt,\int_c^{b} \!\!wf_w^2 \dt\! \Big\}= \infty.$$ By \cite[Theorem 6.4.1]{Z1}, the equation 
	 	$My=0$ has oscillatory solutions both near $a$ and near $b$. Since the existence of an oscillatory solution contradicts the generalized maximum principle, it follows that $M\not\geq 0$ in $(a,b) \setminus K$, but this contradicts our assumption.	
	 \end{proof}
	\section{ Ermakov-Pinney equation and one-dimensional Hardy inequalities}\label{sec_EP_eq}
	In the present section we introduce a new approach to construct Hardy-type inequalities in one-dimension. We  exploit the explicit  general solutions of the well known Ermakov-Pinney semilinear ordinary differential equation, to obtain Hardy-type inequalities for one-dimensional Schr\"odinger operators. Moreover, this method gives rise to an infinite converging series of Hardy-weights in the spirit of \cite{BFT}.
	
	Consider the classical one-dimensional Hardy's inequality
	$$
	\int_0^\infty |\phi'(t)|^2\dt \geq \int_0^\infty \dfrac{|\phi(t)|^2}{4t^2}\dt  \qquad  \forall  \phi\in C_0^\infty(\R_{+}).
	$$ 
	It is known that $w(t):=1/(4t^2)$ is an optimal Hardy-weight for the Laplacian on $\R_+$ with a ground state $f_0(t):=\sqrt{2t}$. \cite{DFP}. Notice the functional  relation between $w$ and the corresponding ground state
	$f_0(t)=\sqrt{2t}=w(t)^{-1/4}$. Equivalently, the ground state $f_0$ solves the Ermakov-Pinney semilinear differential equation \cite{E,PI}
	$$
	-y''=\frac{1}{y^3}  \qquad \mbox{in } (0,\infty).
	$$ 
	Motivating by the above observation, we use the Ermakov-Pinney equation  to construct new 
	{\em strictly positive} Hardy-weights.
		
	Let $q\in L^1_{\text{loc}}(\R_{+})$, and $k>0.$
	By \cite{E,H,PI}, a general positive solution to the Ermakov-Pinney equation
	\begin{equation}\label{Pinney equation}
	-y''+qy=\dfrac{k}{y^3} \qquad \mbox{in } (a,b) , 
	\end{equation}
	is given by
	$$
	y=\sqrt{\left|c_1 v_1^2 +c_2 v_2^2 + 2c_3 v_1 v_2\right|}\, ,
	$$
	where $c_i, i=1,2,3$ are fixed real numbers
	such that $c_3^2-c_1c_2=k$, and $v_i$, $i=1,2$ satisfy the linear Schr\"odinger equation   
	$$
	-u'' +q u=0, \quad \mbox{ such that }\;\;  W(v_1,v_2):= v_1v_2'-v_2v_1'=1 \quad \mbox{in } (a,b).
	$$
	Here $(a,b)$ is  any subinterval of the sets
	\begin{align*}	
&\left 	\{t\in (a,b) \mid \dfrac{v_1(t)}{v_2(t)}\neq \dfrac{-c_3 \pm \sqrt{k}}{c_1}  \right \} & \mbox{if } c_1\neq 0, \\
&\left 	\{t \in (a,b)\mid \dfrac{v_2(t)}{v_1(t)}\neq \dfrac{-c_3 \pm \sqrt{k}}{c_2} \right \} & \mbox{if } c_1=0 ~ \text{and} ~ c_2 \neq 0.
 	\end{align*}

	\begin{lem}\label{optimal pinney} \label{2.15}
		
			Let $q\in L^1_{\loc}(\R_{+})$. Suppose that the equation 
			$$L(y)=-y''+qy=0$$ admits two linearly independent positive  solutions, $v_1$ and $v_2$,  in $(0,b)$ for some $b>0$  such that
			 $v_1$ has minimal growth near $0$, and  $W(v_1,v_2)=v_1'v_2-v_1v_2'=1$.
		Set $$f_w:=\sqrt{\left|c_1 v_1^2 +c_2 v_2^2 + 2c_3 v_1 v_2\right|} ,$$ 
		where $c_i\in \R$ satisfy $c_3^2-c_1c_2=1$.
			Then there exists $0<m \leq b$ such that
			$f_w$ is a (well-defined) positive solution of the equation  
			$$(L-w)y=0 \qquad  \mbox{in } (0,m), $$
			where 
					$$
			w:=\frac{L(f_w)}{f_w}=\dfrac{1}{f_w^4}>0\qquad  \mbox{in } (0,m).
			$$
			 Moreover, for any $0<\varepsilon<m$ 
			 $$  \int_0^{m-\vge} \frac{1}{f_w^2} \dt=
			 \int_0^{m-\vge} w f_w^2 \dt=\infty.$$
	\end{lem}
	\begin{proof}
		Let $0<b_1<b$ be fixed.
		By  reduction of order, given $v_1$, a second local linearly independent positive solution of the equation $Ly=0$ in $(0,b_1)$ is given by 
		\begin{equation}\label{v_2}
		v(t)=v_1(t)\int_t^{b_1} \dfrac{1}{v_1(s)^2} \ds \qquad  0<t<b_1.
		\end{equation}
	Since $v(b_1)=0$  it follows from \cite[Appendix 1]{MU}  that $v(t)$ has minimal growth near $b_1$, and hence, 
		 $v_2=\alpha v_1+\beta v$, where $\alpha \geq 0$ and $\beta>0$.
		
		The assumption that $v_1$ is a positive solution of minimal growth near $0$ of the equation $L y=0$ implies that  $v_1 \leq C_\vge v_2$ and $v_1 \leq C_\vge v$ in $(0,\varepsilon)$ for any $0<\varepsilon <b_1$.
			 Moreover, since $v_1'v_2-v_1v_2'=1 $, it follows  that $v_1/v_2$  is monotone increasing. Therefore, the singular points of $f_w$ are isolated. Consequently,   $f_w$ is well defined in $(0,m),$ where
		 $m=\min \{b_1,t_0-\delta \}$. Here $t_0$ is first zero of $f_w$ and $0<\delta<t_0.$ \\
		Given $0<\varepsilon<m$ let us compute 
		\begin{align*}
		&
		\int_{0}^{m-\varepsilon} \dfrac{1}{f_w^2} \dt =
		\int_{0}^{m-\varepsilon} \dfrac{1}{|c_1 v_1^2+c_2v_2^2+2c_3 v_1v_2|} \dt
		\geq C \int_{0}^{m-\varepsilon}\dfrac{1}{v_2^2} \dt \\[2mm]
		&
		=
		\int_{0}^{m-\varepsilon} \left (\alpha v_1+ \beta v_1\int_t^{b_1} v_1^{-2} \ds \right )^{\!\!\!-2} \dt  \geq 
		 \int_{0}^{m-\varepsilon}\left (\alpha_1 v_1+ \beta v_1\int_t^{m-\varepsilon} v_1^{-2} \ds \right )^{\!\!\!-2} \dt \\[2mm] 
		&	\geq 
		 \int_0^{m-\varepsilon} \left ( \beta_1 v_1\int_t^{m-\varepsilon} v_1^{-2} \ds \right )^{\!\!\!-2}\dt,
		\end{align*} 
		where $\ga_1\geq \ga$ and $\gb_1\geq \gb$.
		The change of variable $\tau:=\int_t^{m-\varepsilon} v_1^{-2} \ds $ implies that 
		$$
		\int_0^{m-\varepsilon} \dfrac{1}{f_w^2} \dt \geq  C\int_0^{M} \dfrac{1}{\tau^2} \dtau =\infty.
		$$
		Moreover,
		$$
		\int_0^{m-\varepsilon} w f_w^2 \dt = \int_0^{m-\varepsilon} \dfrac{1}{f_w^2} \dt=\infty. \qquad \qedhere
		$$
	\end{proof}
	\begin{rem}\label{23}
		\em{
			If  $c_1,c_2$, and $c_3$ are nonnegative,  then for any $0<m<b$
			the function $f_w$ in Lemma \ref{optimal pinney} is well defined in $(0,m)$.  
		}
	\end{rem}
	As a corollary of Lemma \ref{2.15}  we obtain the following result:
	\begin{proposition}\label{2.17}		
		Assume that  $Ly:=-y'' +qy$ is nonnegative in $\R_{+}$, where $q\in L^1_{\mathrm{loc}}(\R_{+})$. Fix $m>0$. Then there exists an infinite sequence of strictly positive functions 
		$\{w_{i}\}_{i\in \N}$   (depending on $m$), such that for any $k\in \N$ the function
			$$
			\tilde{w}_{k}:=\sum_{i=1}^{k}  {w}_{i} 
			$$
			satisfies:
			\begin{itemize}
				\item $M_{k}:=L-\tilde{w}_{k} \geq 0$  in $(0,m)$. 
				\item  There exists a positive solution $y_{k}$ of the equation $M_{k} y=0$ in $(0,m+1)$ such that 
				$$\int_0^m \dfrac{1}{y_{k}^2} \dt =\int_0^m \tilde{w}_{k} y_{k}^2\dt=\infty, $$ and  
				for any $0<\vge<m$, $\lambda_0(L,\tilde{w}_{k},(0,\vge))=1$.
			\end{itemize}
	\end{proposition}
\begin{proof} Fix $m\in (0,\infty)$. 
	Let $v_{1,0}$ and $v_{2,0}$ be two positive solutions of $Ly=0$ in $(0,m+1)$ (depending on $m$) such that their Wronskian $W(v_{1,0}, v_{2,0})$ equals $1$, and such that $v_{1,0}$ has minimal growth in a neighborhood of $0$. In particular, $v_{1,0}$ is positive and well-defined in $\R_{+}$.  Note also that using \eqref{v_2}, it follows that  $v_{2,0}$ is well-defined and positive in $(0,m+1)$.

	Define $w_{1}:=y_{1}^{-4}$, where 
	$$y_{1}:=\sqrt{c_{1,1} v_{1,0}^2+c_{2,1}v_{2,0}^2+2c_{3,1}v_{1,0}v_{2,0}}\,,$$ 
	and $c_{i,1}$, are positive numbers for $i=1,2,3$ satisfying $(c_{3,1})^2-c_{1,1}c_{2,1}=1$.  By Remark \ref{23},  $y_{1}$ is a positive solution of the equation
	$$M_{1}(y):= (L-{w}_{1})y=0 \qquad \mbox{in } (0,m+1).$$ 
	In particular, $M_{1}$ is subcritical in $(0,m+1-\vge_1)$, where $\vge_1=1/2$. 
	
	Repeating the same argument, we obtain in the $j$-th step, $j\geq 2$,   
	\begin{align*}
	&y_{j}:=\!\sqrt{c_{1,j} v_{1, j-1}^2\!+c_{2,j}v_{2, j-1}^2\!+\!2c_{3,j}v_{1, j-1}v_{2, j-1}} \;\mbox{ in } (0,m+1 -\vge_{j-1}), \\
	&\mbox{where } c_{i,j}>0, \mbox{and }(c_{3,j})^2-c_{1,j}c_{2,j}=1.\\
	&{w}_{j}:= (y_{j})^{-4} \quad \mbox{in }  (0,m+1-\vge_{j-1}), \; \;\vge_j:= \vge_{j-1}+(1/2)^j.\\ 
	&v_{1,j}  \mbox{ is a positive solution of $M_ju=0$ in  $(0,m+1 -\vge_{j})$ of minimal growth at } 0.\\
	&v_{2,j}(t):= \!\ga_j v_{1,j}\!+\! \gb_j  v_{1,j}(t)\!\!\int_t^{m+1 -\vge_{j}} \!\!\!\frac{1}{v_{1,j}(s)^2}\! \ds\!>\!0,\; \mbox{ in } (0,m+1 -\vge_{j}),  W(v_{1,j}, v_{2,j})=1,
\end{align*} 
	where $\ga_j, \gb_j>0$. So,  $M_{k}$ is subcritical   in $(0,m)$. 
Moreover, by Lemma \ref{2.15}, we obtain
	\begin{equation*}
\int_0^m y_{k}^2 {w}_{k} \dt = \int_0^{m}\dfrac{1}{y_{k}^2} \dt =\infty ,
	\end{equation*}
	and consequently, 
	\begin{equation*}
	\int_0^m y_{k}^2\tilde{w}_{k} \dt= \infty.
	\end{equation*}
	 Hence, as in Proposition \ref{2.14},  it follows  that
	 $\lambda_0(L,\tilde w_{k},(0,\varepsilon))=1$ for all $\varepsilon \in (0,m)$. 
\end{proof}
\begin{rem}\label{rem_inf-series}
	\em{
		For any $k\in \N$ the inequality $L-\tilde{w}_k \geq 0$ in $(0,m)$ in Proposition~\ref{2.17}  implies the inequality  $$\int_0^m \big(|\vgf'|^2 +q\vgf^2\big)\dt \geq \int_0^m \tilde{w}_k \vgf^2 \dt   \qquad  \forall \in C_0^{\infty}(0,m).$$
	Fix $\vgf\in C_0^{\infty}(0,m)$, it follows that 
	$$\int\limits_0^m  \tilde{w}_k\vgf^2 \dt\leq \int_0^m \big(|\vgf'|^2 +q\vgf^2\big)\dt  \leq\int_0^m \big(|\vgf'|^2 +|q|\vgf^2\big)\dt  \leq C.$$ 
	Since the sequence $\{\tilde{w}_k\}_{k=1}^\infty$ is increasing in $k$ and positive,  it follows by the monotone convergence theorem that 
	$$\tilde{w}:=\lim_{k \to \infty}\tilde{w}_k=\sum_{k=1}^{\infty} {w}_k$$ is finite almost everywhere in $(0,m)$. Moreover, $\tilde{w}$ is a Hardy-weight for $L$ in $(0,m)$.
}
\end{rem}
	\section{Improved and optimal Hardy-weights in higher dimensions}\label{sec-optimal_higher_d}
		In the present section we  reduce the problem of constructing optimal Hardy-weights in higher dimensions to a one-dimensional problem. Roughly speaking, each optimal Hardy-weight for the one-dimensional Laplacian yields in the symmetric case optimal Hardy-inequalities for subcritical operators in higher dimensions.  
	\begin{defi}
		\em{
			We say that a function $0\leq w\in L^{\infty}_{\text{loc}}(\R_{+})$ belongs to $\mathcal{W}_1(\R_{+})$ if $w$ is an optimal Hardy-weight for the one-dimensional Laplace operator
			$$
			Ly:=-y''\qquad \mbox{on } \R_{+}.
			$$
		}
	\end{defi}
We have.
	\begin{lem}\label{1D if and only if}
			Let $0\leq w\in L^1_{\mathrm{loc}}(\R_{+}).$ Then $w\in \mathcal{W}_1(\R_{+})$ with ground state $f_w$ if and only if the following three conditions are satisfied:
			\begin{enumerate}
				\item There exists $f_w >0$ such that $(L-w)f_w=0$ in $\R_{+}$,\\[1mm]
				
				\item $\displaystyle{\int_{0}^{1}\dfrac{1}{f_w^2}\dt = \int_{1}^{\infty}\dfrac{1}{f_w^2}\dt=\infty}$,\\[2mm]
					
				\item $\displaystyle{\int_{0}^{1}w  f_w^2\dt=\int_{1}^{\infty}w f_w^2\dt=  \infty}$,
			\end{enumerate}
	\end{lem}
	\begin{proof}
	Follows immediately from Proposition \ref{prop_2.10}. 
	\end{proof}
	Our main result of the the present section is as follows.
	\begin{theorem}\label{criticality_linear} 
			Let  $\Omega \subset \mathbb{R}^n$, $n \geq 2$,  be a domain containing $x_0$. Let $P$ be a subcritical  operator  in $\Omega$ and let $G(x):=G_{P}^{\Omega}(x,x_0)$ be its positive minimal Green function with singularity at $x_0$. Let $0\lneqq \varphi \in \core$  and consider the associated  Green potential 
			$G_\varphi(x)$ (given by \eqref{green potential}).
			Assume that there exists $u\in \mathcal{C}_P(\Omega)$  satisfying
			\begin{equation}\label{eq:Gphi/u_tends_0}
			\lim_{x \to \overline{\infty}} \dfrac{G(x)}{u(x)}=0.
			\end{equation}
			Fix $w\in \mathcal{W}_1(\R_{+})$, and let $f_w(t)$ be the corresponding  ground state. 
			Suppose that $f_w'\geq 0$ on  $\{t=G_{\varphi}(x)/u(x) \mid x\in \supp(\varphi) \}$.
			Define
			\begin{equation*}
			W:=\dfrac{P \left (u f_w \left ( G_{\varphi} /u\right) \right )}{u f_w \left ( G_{\varphi} /u\right)}\, .
			\end{equation*}
			Then 
			\begin{enumerate}
				\item 
				 $W\geq 0$ in $\Gw$ and $W:=\left |\nabla (G_{\varphi}/u) \right |_{A}^{2}w(G_{\varphi}/u), ~ \text{in} ~ \Gw \setminus \supp(\varphi).$
				 
				\item			$P-W$ is critical  in $\Omega$ with ground state $u f_w (G_{\varphi}/u)$. 
				\item $W$ is optimal in a neighborhood of infinity in $\Omega$. 
			\item If $P$ is symmetric, then $P-W$ is null-critical with respect to $W$,  and therefore, $W$ is an optimal Hardy-weight for $P$ in $\Gw$.
		\end{enumerate}
	\end{theorem}
	\begin{proof}
		 Without loss of generality, we assume that $G_{\varphi}/u<1$ in $\Gw$.  By Lemma \ref{1D if and only if}, the function $$f_1(t):=f_w(t)\int\limits_{1}^{t}\dfrac{1}{f_w^2(s)}\ds$$
		is a linearly independent solution of the equation $(L-w)y=0$ in $\R_{+}$ which is positive for $t>1$, negative for $t<1$, and satisfies
		\begin{equation}\label{eq:3.3}
		\lim_{t\to 0}\dfrac{f_w(t)}{f_1(t)}=\lim_{t\to \infty}\dfrac{f_w(t)}{f_1(t)}=0.
		\end{equation}
			
		Consider the functions $h,v:\Gw\to \R$ given by 
		$$v(x):=u(x)f_w \left(\dfrac{G_{\varphi}(x)}{u(x)}\right),  \qquad 
		h(x):=
		u(x)f_1 \left(\dfrac{G_{\varphi}(x)}{u(x)}\right).
		$$
		A direct calculation (see (4.13) in \cite{DFP}) shows that
		for any $u\in \mathcal{C}_P(\Gw)$ and $g\in \mathcal{S}_P(\Gw)$
		\begin{equation}\label{Puv}
		P(uf_w(g/u))=-uf_w''(g/u)|\nabla (g/u)|_A^2+uf_w'(g/u)P(g).
		\end{equation}
		In particular, $W \geq 0$ in $\Gw$,
		and $v\in \mathcal{C}_{P-W}(\Gw)$. Moreover, 
		\begin{align*}
		&
		P(v)\!=\!-|\nabla (G_{\varphi}\!/u)|_{A}^{2}uf_w''(G_{\varphi}\!/u)\!=\!
		|\nabla (G_{\varphi}\!/u)|_{A}^{2}w(G_{\varphi}\!/u)uf_w(G_{\varphi}\!/u)\!=\!
		Wv  \text{ in }  \Gw \!\setminus\! \supp(\varphi), \\[2mm] 
		&
		P(h)\!=\!-|\nabla (G_{\varphi}\!/u)|_{A}^{2}uf_1''(G_{\varphi}\!/u)\!=\!
		|\nabla (G_{\varphi}\!/u)|_{A}^{2}w(G_{\varphi}\!/u)uf_1(G_{\varphi}\!/u)\!=\!Wh  \text{ in }  \Gw \!\setminus \!\supp(\varphi).
		\end{align*}
		Hence, $h,v\in \mathcal{C}_{P-W}(\Omega\setminus \supp(\varphi))$.
		Moreover,  \eqref{eq:Gphi/u_tends_0} and \eqref{eq:3.3} imply
		\begin{equation*}
		\lim_{x\to \overline{\infty}}\dfrac{v(x)}{h(x)}=\lim_{x\to \overline{\infty}}\dfrac{u(x)f_w\left(\dfrac{G_{\varphi}(x)}{u(x)}\right)}{u(x)f_1 \left(\dfrac{G_{\varphi}(x)}{u(x)}\right)}=0.
		\end{equation*}
		By  a Khas'minski\u{i}-type criterion \cite[Proposition~6.1]{DFP},  it follows that $v$ has minimal growth in a neighborhood of infinity in $\Omega$, implying that $v$ is a ground state of the operator $P-W$ in $\Gw$, and by Theorem~\ref{thm_crit}, $P-W$ is critical in $\Gw.$ 
		
		Next we prove that $W$ is optimal at $\overline{\infty}$ (recall that $P$ is not necessarily symmetric). Let $\xi>0$, and consider the equation
		\begin{equation} \label{330}
		-y''-wy=\xi w y, ~ \text{in} ~ \R_{+}.
		\end{equation}
		Lemma \ref{1D if and only if} implies that 
		\begin{equation*}
		\int_{0}^1 w f_w^2  \dt=\infty.
		\end{equation*}
		By Proposition~\ref{2.14} (or \cite[Theorem 6.4.1]{Z1}), there exists an oscillatory (near $0$) solution $y_{\xi}$ to \eqref{330} satisfying $y_\xi,y_\xi'\in AC_{\text{loc}}(0,1)$. Therefore, $u y_{\xi}(G_{\varphi}/u)$ is an oscillatory solution of the operator $P-(1+\xi) W$ in a neighborhood of $\overline{\infty}$. Hence, $P-(1+\xi) W$ does not have a positive solution in a neighborhood of $\overline{\infty}$. In other words, $\lambda_{\infty}=\lambda_0=1$. Thus, $W$ is optimal at $\overline{\infty}.$
		
		It remains to prove that when $P$ is symmetric, $P-W$ is null-critical with respect to $W$.
		
		Take $\ga >0$ sufficiently small such that 
		$$
		\{0<G_{\varphi}/u<\ga \}\cap \supp \varphi=\emptyset.
		$$
		For any $0<\varepsilon<\ga $, the coarea formula (see (9.4)) in \cite{DFP}) implies
		\begin{align}\label{eq-coarea}
		& \int\limits_{\varepsilon<G_{\varphi}/u<\ga } v^2 W \dx=\!\!\! \int\limits_{\varepsilon<G_{\varphi}/u<\ga } \!\!\!u^2f_w^2(G_{\varphi}/u) |\nabla (G_{\varphi}/u)|_{A}^2w(G_{\varphi}/u) \dx = \nonumber\\ 
		&
		\int\limits_{\varepsilon}^{\ga}f_w^2(t)w(t)  \int\limits_{G_{\varphi}/u=t} u^2|\nabla (G_{\varphi}/u) |_A \dsigma  \dt 
		=  \nonumber\\ &
		\int\limits_{\varepsilon}^{\ga}f_w^2(t)w(t)  \int\limits_{G_{\varphi}/u=t} \langle
		uA \nabla G_{\varphi}-G_{\varphi}A \nabla u,\vec{\sigma}\rangle 
		\dsigma   \dt, 
		\end{align}
		where $\vec{\sigma}$ is a normal vector (in the metric $|\cdot  |_A$) to the level set $\{ G_{\varphi}/u=t\}.$
		By Green's formula for  $0<\varepsilon<t_1<t_2<\ga $ such that ${t_1<G_{\varphi}/u<t_2}$ is a regular domain (see \cite{DFP})), we have 
		\begin{align*}
		&
		0=\int\limits_{t_1<G_{\varphi}/u<t_2} (P(u)G_{\varphi} - P(G_{\varphi}) u) \dx
		= \\ &
		\int\limits_{G_{\varphi}/u=t_2} \langle
		uA \nabla G_{\varphi}-G_{\varphi}A \nabla u,\vec{\sigma}\rangle \ds-
		\int\limits_{G_{\varphi}/u=t_1} \langle
		uA \nabla G_{\varphi}-G_{\varphi}A \nabla u,\vec{\sigma}\rangle \ds.
		\end{align*}
		Hence, there exists $\gb$ such that for any $t\in (0,\ga )$
		$$
		\int_{G_{\varphi}/u=t} \langle
		uA \nabla G_{\varphi}-G_{\varphi}A \nabla u,\vec{\sigma}\rangle 
		\ds   =\gb. 
		$$
		Moreover, by \eqref{eq-coarea}, we infer that $\gb > 0$. Consequently, 
		$$
		\int_{0<G_{\varphi}/u<\ga} v^2 W \dx = \gb \lim_{\varepsilon \to 0}\int_{\varepsilon}^\ga
		f_w^2(t)w(t) \dt = \infty . \qquad  \qedhere
		$$
	\end{proof}
\begin{rem}
	\em{
	If condition \eqref{eq:Gphi/u_tends_0} is satisfied in $\Gw$ for some $u\in \mathcal{C}_P(\Omega)$, then  \eqref{eq:Gphi/u_tends_0} is satisfied  in $\Gw^*:=\Gw \setminus\{x_0\}$ for any $x_0\in \Gw \setminus \supp (\varphi)$ with $\tilde{u}:=u+G$ replacing $u$. Indeed,  one may
take the pair $G_{\varphi}$ and $\tilde{u}$.  In this case, the corresponding (optimal in the symmetric case) Hardy-weight for $P$ in $\Gw^*$ is    $W:=\dfrac{P \left ( \tilde{u} f_w \left ( G_{\varphi} /\tilde{u}\right ) \right )}{\tilde{u} f_w \left ( G_{\varphi} /\tilde{u}\right)}$\,.
}
\end{rem}
As a corollary of Proposition \ref{24} and Theorem \ref{criticality_linear}, we obtain the following result which generalizes the results of Barbatis, Filippas and Tertikas \cite{Barbatis}: 
\begin{lem}\label{lemma 5.5}
	Let  $\Omega \subset \mathbb{R}^n$, $n \geq 2$,  be a domain containing $x_0$. Let $P$ a subcritical  operator  in $\Omega$, and let $G(x):=G_{P}^{\Omega}(x,x_0)$ be its Green function with singularity at $x_0$,  and let $G_{\varphi}$ be a Green potential with  density $0\lneqq \varphi \in \core$  (see \eqref{green potential}).
	Then there exists an infinite sequence of Hardy-weights
	 $\{ {W}_i \}_{i\in \N}$ such that for any $k\in \N$ 
	the function 
	$$
	 \tilde{W}_k:=\sum_{i=1}^{k} {W}_i 
	$$
	satisfies $P-\tilde{W}_k \geq 0$ in $\Gw$ and $\lambda_0(P,\tilde{W}_k,\Gw)=\lambda_\infty(P,\tilde{W}_k,\Gw)=1$. Moreover, $\tilde W:=\sum_{k=1}^{\infty} {W}_k$ is a Hardy-weight for $P$ in $\Gw$.
\end{lem}
\begin{proof}
	Apply Proposition \ref{2.17} with $q=0$ to obtain a sequence $\{w_i\}_{ i \in \N}$ such that 
	$$-y''-\sum_{i=1}^k w_i \geq 0 ~\mbox{in} ~ (0,\sup\limits_{\Gw} (G_{\varphi}/u)).$$
	By the proof of Theorem~\ref{criticality_linear} the function
	$\tilde{W}_k:=|\nabla(G_{\varphi}/u)|_A^2\sum_{i=1}^k {w}_i(G_{\varphi}/u)$  is a Hardy-weight of $P,$ i.e., $P-\tilde{W}_k \geq 0$ in $\Gw$. Moreover, as in Remark~\ref{rem_inf-series}, it follows that $\tilde W$ is a well defined Hardy-weight for $P$ in $\Gw$. 
	
	Assume by contradiction that  $\lambda_0(P,\tilde{W}_k,\Gw \setminus K)>1$, namely, there exists $K \Subset \Gw$ such that 
	$P-(1+\varepsilon^2)\tilde{W}_k \geq 0$ in $\Gw \setminus K$ for some $\varepsilon>0$. By the same argument as in Lemma~\ref{2.14} there exists an oscillatory solution of the equation $-y''-(1+\varepsilon^2)(\sum_{i=1}^k  {w}_i)y=0$ near $0$, implying that there exists an oscillatory solution of the equation $P-(1+\varepsilon^2)\tilde{W}_k=0$ near $\overline{\infty}$. Such a solution contradicts the maximum principle. Hence,  
	$\mathcal{C}_{P-(1+\varepsilon^2)\tilde{W}_k}(\Gw \setminus K)= \emptyset$, and we arrive at a contradiction. 
\end{proof}
	\section{Improved optimal Hardy-weights in the nonsymmetric case}\label{sec-improved-optimal}
	In the previous section we proved for symmetric operators that the Hardy-weights obtained in Theorem~\ref{criticality_linear} are optimal, while in the nonsymmetric case we could only prove the criticality and optimality at infinity.  In the present section we use the Ermakov-Pinney equation to obtain a family of optimal Hardy-weights for not necessarily symmetric operators.

	The following lemma exploits the general positive solutions of the Ermakov-Pinney equation $y''(t)=-\dfrac{1}{y^3}$ in $(0,2/a)$ that vanishes at $t=0$, namely,  $y_a(t):=\sqrt{2t-at^2} $ to obtain optimal Hardy-weights. 
Using Proposition \ref{prop_2.10} we obtain:
	\begin{lem} \label{3pinney}
			Let $w(t):=(2t-at^2)^{-2}$ and $f_w(t):=\sqrt{2t-at^2}$, where $a>0$ is fixed. Then 
			\begin{enumerate}
				\item $-f_w ''-w f_w=0$, and  $f_w>0$ in $ (0,2/a)$. \\
				
				\item $\displaystyle{\int_{0}^{1/a}\dfrac{1}{f_w^2}\dt= \int_{1/a}^{2/a}\dfrac{1}{f_w^2}\dt=\infty}$,\\[1mm]
				
				\item $\displaystyle{\int_{0}^{1/a}w f_w^2\dt=\int_{1/a}^{2/a}w f_w^2\dt} =  \infty$.
			\end{enumerate} 
	In particular, $w$ is an optimal Hardy-weight for the operator $Ly=-y''$ in $(0,2/a)$.
	\end{lem}
	\begin{rem} \em{
			 We exclude the case $a\leq 0$ since we are interested in Hardy-weights  greater than the classical Hardy-weight $w_{\mathrm{class}}(t)=(2t)^{-2}.$
			}
	\end{rem}
	The following proposition is an important ingredient in the construction of an improved optimal weight for the nonsymmetric case. 
	\begin{proposition}\label{SL_properties}
		
			Let $w(t)=(2t-at^2)^{-2}$. Then for any $\xi,M,a>0$, and $t\in (0,2/a )$ the function 
			$$u_{\xi}(t):=\sqrt{2t-at^2}\cos\left(\dfrac{\xi}{2}\log \left (\dfrac{Mt}{2-at} \right )\right)$$ satisfies the following properties:
			\begin{enumerate}
				\item $ -u_{\xi}''-(1+\xi^2)wu_{\xi}=0 \quad \mbox{in } (2/(M\mathrm{e}^{\pi/\xi}+a), 2/(M+a))$,\\
				\item The boundary condition: $u_\xi ' (2/(M+a))=\dfrac{M^2-a^2}{4M}u_\xi (2/(M+a))$,\\[1mm]
				\item The boundary condition: $u_{\xi}\! \left(2/(M\mathrm{e}^{\pi/\xi}+a)\right)=
				u_{3\xi}\! \left(2/(M\mathrm{e}^{\pi/\xi}+a)\right)=0$,
				\item $u_{\xi}$ converges pointwise to $\sqrt{2t-at^2}$ as $\xi \to 0$,
				\item $| u_{\xi}(t)| \leq \sqrt{2t-at^2}.$
			\end{enumerate}
	\end{proposition}
	\begin{proof}
		The assertions of the proposition can be verified straightforwardly, see Appendix for a method to obtain the solutions $u_\xi$. 
	\end{proof}
	We proceed to find a family of optimal Hardy-weights in $\Gw$ greater (in a neighborhood of $\infty$) than $W_{\mathrm{class}}$, the classical optimal weight obtained in \cite{KP} which satisfies  
	$$
	W_{\mathrm{class}}=\dfrac{|\nabla(G_\varphi /u)|_{A}^2}{4(G_{\varphi} /u)^2} \qquad \mbox{in } \Gw \setminus \supp(\varphi).
	$$
	
	\begin{thm}\label{3.7} 
			Let  $\Omega \subset \mathbb{R}^n$, $n \geq 2$,  be a domain containing $x_0$. Let $P$ be a subcritical  linear operator  in $\Omega$ and let $G(x):=G_{P}^{\Omega}(x,x_0)$ be its  minimal positive Green function with singularity at $x_0$,  and let $G_{\varphi}$ the Green potential with  density $0\lneqq \varphi \in \core$  (see \eqref{green potential}). Assume that there exists $u\in \mathcal{C}_P(\Omega)$  satisfying
			\begin{equation}\label{eq:G/u_tends_1}
			\lim_{x \to \overline{\infty}} \dfrac{G(x)}{u(x)}=0.
			\end{equation}
			Let $$0<a<\dfrac{1}{\sup\limits_{\Gw}{(G_{\varphi} /u)}}\, ,\quad
			f_w(t):=\sqrt{2t-at^2}.
			$$ 
			Then
			$$
			W:=\frac{P \left (uf_w \left ( G_{\varphi} /u\right) \right )}{uf_w \left ( G_{\varphi} /u\right)}
			$$
			is an optimal Hardy-weight of $P$ in $\Gw$, and
			$$W=\frac{|\nabla(G_{\varphi}/u)|_{A}^2}{(G_{\varphi} /u)^2(2-a(G_{\varphi}/u))^2}\geq W_{\mathrm{class}}:=\frac{|\nabla(G_\varphi /u)|_{A}^2}{4(G_{\varphi} /u)^2} \quad  \text{in }  \Gw \setminus \supp(\varphi).$$
		
	\end{thm}
	\begin{proof} 
		Obviously, $aG_{\varphi}/u < 1$,
		therefore, the function $W$ is well-defined and by \eqref{Puv} $W$ is nonnegative in $\Gw$ since $f'_w(t)>0$ for $0<at<1$.
		
		We first prove the criticality of $P-W$ in $\Gw$.
		Let $f_1(t):=f_w(t)\int\limits_{t}^{1} (f_w)^{-2}\ds$. Since
		$$
		\lim_{t \to 0} \frac{f_w(t)}{f_1(t)}=0,
		$$
		and $uf_1(G_{\varphi}/u)$ is a positive solution of $P-W$ in a neighborhood of $\bar \infty$,
		it follows the Khas'minski\u{i} criterion \cite[Proposition~6.1]{DFP}  that  the function $\phi_0:=uf_w(G_{\varphi}/u)$ is a positive solution of the equation  $(P-W)v=0$ in $\Omega$ having minimal growth   in a neighborhood of infinity in $\Omega$, hence, $P-W$ is critical in $\Gw$.
	
	\medskip	
	
		Next we prove the null-criticality of $P-W$ with respect to $W$.
		Fix a sufficiently large constant $M>0$  such that 
		\begin{equation}\label{zar}
		\left \{x\in\Gw \left|   \; 0<G_{\varphi}/u<2/M\right. \right  \}\cap
		\supp(\varphi)=\emptyset.
		\end{equation}
		Let $\phi_0$ and $\phi_0^*$ be the ground-states of $P-W$ and $P^* -W$ respectively. 
		Let 
		$$\Gw_{\xi}:=\left \{x\in\Gw \left|   \; \frac{2}{M\mathrm{e}^{\pi/\xi }+a }<\frac{G_{\varphi}(x)}{u(x)}<\frac{2}{M+a} \right. \right \}, $$
$$\Gg:=\left\{ x\in \Gw\mid \frac{G_{\varphi}(x)}{u(x)}=\dfrac{2}{M+a} \right\}, \;\;
\Gg_\xi:=\left\{ x\in \Gw\mid \frac{G_{\varphi}(x)}{u(x)}= \frac{2}{M\mathrm{e}^{\pi/\xi }+a }\right\},$$
and assume for the moment that $\Gg$ and $\Gg_\xi$ are regular hypersurfaces. 
 Let $\phi^{*}_{\xi} $ be the positive solution of the Dirichlet problem:
		\begin{equation}\label{phi-xi*}
		\begin{cases}
		(P^*-W)u=0 & \text{in } \Gw_{\xi}, \\[3mm]
		u(x)=\phi_{0}^* & \text{on } \Gg, \\[3mm]
		u(x)=0 & \text{on }   \Gg_\xi . 
		\end{cases}
		\end{equation}
		Let $$\phi_{\xi}: =u u_{\xi}(G_{\varphi} /u),$$
		where $$u_{\xi}(t)=\sqrt{2t-at^2} \cos \left(\dfrac{\xi}{2} \log \left(\dfrac{Mt}{2-at} \right )\right ).$$
		By \eqref{zar}, $\Gw_{\xi}\cap \supp(\varphi)=\emptyset$, and therefore, 
		$$
		P(\phi_{\xi})=(1+\xi^2)W \phi_{\xi} \qquad  \text{in }  \Gw_{\xi}.
		$$
		Moreover, $|\phi_{\xi}| \leq \phi_0$, and $\phi_{\xi} \to \phi_{0}$ pointwise in $\Gw$  as $\xi \to 0$.
		In addition,  by \cite[Theorem 8.2]{DFP}, $0<\phi_{\xi}^* \leq \phi_0^*$ in $\Gw_\xi$, and  as $\xi \to 0$,  $\phi_{\xi}^* \to \phi_{0}^*$ pointwise in  
		$$\tilde{\Gw}:=\left\{x\in \Gw \mid  \frac{G_{\varphi}(x)}{u(x)}<\dfrac{2}{M+a}  \right\}.$$
		 
		Lemma \ref{SL_properties} and Green's formula imply 
		\begin{align}
		\xi^2\int_{\Gw_\xi} W \phi_{\xi} \phi_{\xi}^*&= 
		\int_{\Gw_{\xi}}\phi_{\xi}^*(P-W)\phi_\xi \dx= B.T(\phi_{\xi},\phi_{\xi}^*),
	 \nonumber\\ 
	 9\xi^2\int_{\Gw_\xi}  W  \phi_{3\xi} \phi_{\xi}^* \dx &= \int_{\Gw_{\xi}}\phi_{\xi}^*(P-W)\phi_{3\xi} \dx=
		B.T(\phi_{3\xi},\phi_{\xi}^*), \label{25}
		\end{align}
		where the boundary terms are
		\begin{align*}
		B.T(\phi_{\xi},\phi_{\xi}^*)&=
		\int_\Gg \,\langle
		\phi_\xi^*A\nabla\phi_\xi -\phi_\xi A\nabla \phi_{\xi}^* + \phi_\xi \phi_\xi^*(\bb -\bt), \vec{\sigma} 
		\rangle \text{d}\sigma,\\
		B.T(\phi_{3\xi},\phi_{\xi}^*)&=
		\int_\Gg \,\langle
		\phi_\xi^*A\nabla\phi_{3\xi} -\phi_{3\xi} A\nabla \phi_{\xi}^* + \phi_{3\xi} \phi_\xi^*( \bb -\bt), \vec{\sigma} 
		\rangle \text{d}\sigma,
			\end{align*}
			and where $\vec{\sigma}$ is a normal vector (in the metric $|\cdot  |_A$) to $\Gg$.  Using the boundary condition on $\Gg$ (see Proposition \ref{SL_properties} and \eqref{phi-xi*}), we have
			\begin{align*}
				&  B.T(\phi_{\xi},\phi_{\xi}^*)\!= \!\int_\Gg \langle
		\phi_0 \phi_0^* A\nabla u \!+\!\frac{M^2-a^2}{4M}u \phi_0\phi_0^* \nabla(G_{\phi}/u)\!-\!\phi_0A\nabla \phi_{\xi}^*  \!+\! \phi_0 \phi_0^*(\bb \!-\!\bt), \vec{\sigma} 
		\rangle \text{d}\sigma\\ &
		= B.T(\phi_{3\xi},\phi_{\xi}^*),
	\end{align*}
	where the last equality follows since the integrand  in the above expression is independent of the choice of either  $\phi_\xi$ or $\phi_{3\xi}$.

	 Hence, by \eqref{25} we have
	\begin{equation}\label{24}
	9\xi^2\int_{\Gw_\xi}  W \phi_{3\xi} \phi_{\xi}^* \dx=
	\xi^2\int_{\Gw_\xi}  W  \phi_{\xi} \phi_{\xi}^* \dx .
	\end{equation}
	For nonregular $\Gw_\xi$, we obtain \eqref{24} by an approximation argument.
	
		Assume by contradiction  that $P-W$ is positive-critical (i.e., $ \int_{\tilde \Gw} W \phi_0 \phi_0^* \dx<\infty$). Dividing \eqref{24} by $\xi^2$ and
		letting 
		$\xi \to 0$, the dominated convergence theorem implies
		$$
		9\int_{\tilde \Gw} W\phi_0 \phi_0^* \dx=\int_{\tilde \Gw} W \phi_0 \phi_0^* \dx,
		$$
		which is a contradiction. 
	\end{proof}
\begin{rem}
	\em{
	The constant  $a$ in Theorem \ref{3.7} was arbitrarily chosen such that $G_{\varphi}/u<1/a$ in $\Gw$. In fact, we may choose any constant $a>0$ satisfying $G_{\varphi}/u \leq 1/a$ in $\Gw$.
}
\end{rem}
Finally, using the method in \cite{Agmon, DFP}, we  establish a  new Rellich-type inequality involving the difference between the Hardy-weight of Theorem~\ref{3.7} and the ``classical" one given by Theorem~\ref{thm:DFP}.   
\begin{theorem}
	Let  $\Omega \subset \mathbb{R}^n$, $n \geq 2$,  be a domain containing $x_0$. Let $P$ be a symmetric subcritical  linear operator  in $\Omega$ and let $G(x):=G_{P}^{\Omega}(x,x_0)$ be its  minimal positive Green function with singularity at $x_0$. Consider  a Green potential   $G_{\varphi}$  as in \eqref{green potential} and  $u\in \mathcal{C}_P(\Omega)$ satisfying \eqref{Ancona}. Set
	$$0<a<\dfrac{2}{\sup\limits_{x\in \Gw}{\left(G_{\varphi}(x) /u(x)\right)}}\, ,$$
	and consider the ``classical" Hardy-weight and the Hardy-weight  given by Theorem \ref{3.7}, respectively, i.e.,    
$$
W_{\mathrm{class}}=\dfrac{|\nabla(G_{\varphi}/u)|_A^2}{4(G_{\varphi}/u)^2}\,, \quad W=\dfrac{|\nabla(G_{\varphi}/u)|_A^2}{(G_{\varphi}/u)^2(2-a(G_{\varphi}/u))^2}\,.
$$
 Then for any $\psi\in C_0^{\infty}(\Gw\setminus \supp(\varphi))$ the following Rellich-type inequality holds true
		\begin{align*} 
\int_{\Gw\setminus \supp(\varphi)}\dfrac{(P(\psi))^2(G_{\varphi}/u)}{W-W_{\mathrm{class}}} \dx
 \geq 
 	 \int_{\Gw \setminus \supp(\varphi)}
\psi^2 (W-W_{\mathrm{class}})(G_{\varphi}/u) 
 \dx.
\end{align*}	
\end{theorem}
\begin{proof}
	Assume first that $P({\bf 1})=0$ and let $u={\bf 1}$, and assume that $\lim_{x \to \overline{\infty}} G(x)=0$.
Recall that $W >W_{\mathrm{class}}$ in $C_0^{\infty}(\Gw\setminus \supp(\varphi))$. Consequently, for $\psi\in C_0^{\infty}(\Gw\setminus \supp(\varphi))$, we have (see (10.9) in  \cite{DFP}),
	\begin{align*}
	&
	\int _{\Gw\setminus \supp(\varphi)}P(\psi)\psi G_{\varphi} \dx= \\[2mm] &
	\int_{\Gw\setminus \supp(\varphi)}\!\!\!\!\!\!\!\! P(\psi\sqrt{G_{\varphi}})\psi\! \sqrt{G_{\varphi}}\! \dx +
	\!\frac{1}{2}\!\int_{\Gw\setminus \supp(\varphi)}\!\!\!\!\!\!\!\psi^2 P(G_{\varphi})\! \dx
	-\!\!\int_{\Gw \setminus \supp(\varphi)}\!\!\!\!\!\!\!  P(\sqrt{G_{\varphi}})\psi^2 \sqrt{G_{\varphi}} \!\dx= \\[2mm] & 
		\int_{\Gw \setminus \supp(\varphi)} P(\psi\sqrt{G_{\varphi}})\psi \sqrt{G_{\varphi}} \!\dx 
			- \int_{\Gw \setminus \supp(\varphi)}  \psi^2 W_{\mathrm{class}}G_{\varphi}\!\dx   
	\geq \\[2mm]&
	\int_{\Gw \setminus \supp(\varphi)} \psi^2 ( W-W_{\mathrm{class}} ) G_{\varphi} \dx \geq 0 ,
	\end{align*}
	where we used the inequality $P-W\geq 0$ to derive the fourth line.
	The Cauchy-Schwartz inequality implies
	$$
	\left(\!\int _{\Gw\setminus \supp(\varphi)}\!\!\!\!P(\psi)\psi G_{\varphi} \!\dx \!\!\right)^{\!\!2}\! \!\leq\!
	\left(\!\int_{\Gw\setminus \supp(\varphi)}\!\!\dfrac{(P\psi)^2G_{\varphi}}{(W\!-\!W_{\mathrm{class}}) } \!\dx\!\right)
	\!\! \left(\!\int_{\Gw\setminus \supp}\!\! \!\!\!\psi^2 ( W\!-\!W_{\mathrm{class}} )G_{\varphi}\! \dx \!\right)\!,
	$$
	and therefore, 
	\begin{align} \label{phitransform}
		\int_{\Gw\setminus \supp(\varphi)}\dfrac{(P\psi)^2G_{\varphi}}{(W-W_{\mathrm{class}})} \dx 
		 \geq 
		 \int_{\Gw \setminus \supp(\varphi)} \psi^2 ( W-W_{\mathrm{class}} ) G_{\varphi} \dx.	
	\end{align}
	Now let $P$ be a general subcritical operator in $\Gw$ which is symmetric in $L^2(\Gw,\!\dx)$, and consider  the operator $P_u:={u}^{-1}Pu$ which is the  ground state transform of the operator $P$ (see for example, \cite[Section 4.1]{DFP}). Note that $P_u$ is symmetric in $L^2(\Gw,u^2\!\dx)$, and a Hardy-weight $W$ for $P$ corresponds  to the Hardy-weight $Wu^2$ for $P_u$.
	Moreover, the corresponding positive minimal Green function is given by  $G_{P_u}^{\Gw}(x,y)=G_{P}^{\Gw}(x,y)u(y)/u(x)$. Hence, for any $\varphi\in C_0^{\infty}(\Gw)$, we have
	$(G_{P_u}^{\Gw})_{\frac{\varphi}{u}}=(G_P^{\Gw})_{\varphi}/u$.

	Let $\psi:=u\vartheta$ where , $\vartheta\in \core$.  Recall that $P_u1=0$, consequently, using the obtained Rellich inequality for $P_u$, we get
	\begin{align*}
&	\int_{\Gw\setminus \supp(\varphi)}\!\!\!\dfrac{(P(\psi))^2(G_{\varphi}/u)}{W-W_{\mathrm{class}}} \!\dx=
	 \!\!\int_{\Gw\setminus \supp(\varphi)}\!\!\dfrac{(P_u(\vartheta))^2(G_{P_u}^{\Gw})_{\frac{\varphi}{u}}}{u^2(W-W_{\mathrm{class}})}\!\dx
	\geq \\[2mm]
&	\int_{\Gw \setminus \supp(\varphi)}\!\!\!
	(u\vartheta)^2 (W-W_{\mathrm{class}})(G_{P_u}^{\Gw})_{\frac{\varphi}{u}} 
	\!\dx=\!\!
	\int_{\Gw \setminus \supp(\varphi)}\!\!\!
	\psi^2 (W-W_{\mathrm{class}})(G_{\varphi}/u) \!\dx. \qedhere
	\end{align*}	
	\end{proof}
\section{Examples}
We illustrate our results with two explicit examples. 
\begin{example}
	\em{
Let $\Gw=\R^n, n \geq 3$, and consider the classical Hardy inequality
$$
\int\limits_{\R^n\setminus\{ 0\}}|\nabla u|^2 \dx \geq \left ( \dfrac{n-2}{2}\right )^2
\int\limits_{\R^n \setminus \{ 0\}}\dfrac{1}{4|x|^2}u^2 \dx \qquad  \forall u\in C_0^{\infty}(\R^n \setminus \{ 0\}).
$$
Let $G(x,y)=C_n|x-y|^{2-n}$ be the Green's function of $-\Delta$ in $\R^n$  and let $G(x):=G(x,0).$ It follows that  
$$
\left (\dfrac{n-2}{2} \right )^2\dfrac{1}{|x|^2}= \dfrac{1}{4}\left |\dfrac{\nabla G(x)}{G(x)} \right |^2
$$
is an optimal Hardy weight in $\R^n \setminus \{ 0\}$.

 Consider the classical Hardy-weight  for $-\Delta$ in $\R^n$, namely $W=-\Gd G_{\varphi}/G_\varphi$, where $G_{\varphi}$ is given by \eqref{green potential}. We have 
$$
W= \dfrac{1}{4}\left | \frac{\nabla G_{\varphi}}{G_{\varphi}}\right |^2 \qquad \mbox{in } \Gw \setminus \supp(\varphi).
$$
 Assume further that $\varphi$ is  rotationally invariant. Clearly, there exists $ C>0$ such that
$G_{\varphi}(x)=CG(x)$ for all $x\notin \supp(\varphi)$ (see for example,\cite[Theorem 9.7]{LL}).
Therefore,
$$
\dfrac{1}{4}\left | \frac{\nabla G_{\varphi}(x)}{G_{\varphi}(x)}\right |^2= 
\dfrac{1}{4}\left | \frac{\nabla G(x)}{G(x)}\right |^2 =\left (\dfrac{n-2}{2} \right )^2\dfrac{1}{|x|^2}  \qquad \mbox{in } \Gw \setminus \supp(\varphi). 
$$
Consequently, Theorem \ref{3.7} implies that there exists $a>0$ and an optimal  Hardy-weight $W$ for $-\Gd$ in $\R^n$  satisfying
	$$W(x)=\dfrac{|\nabla G_{\varphi}(x)|^2}{(2G_{\varphi}(x)-a(G_{\varphi}(x))^2)^2} >  \left (\dfrac{n-2}{2} \right )^2\dfrac{1}{|x|^2} \qquad  \text{in }  \Gw \setminus \supp(\varphi).$$
}
\end{example}
\begin{example}
\em{
	Let $c_1>0$, $c_2 \geq 0$ and $c_3=\sqrt{1+c_1c_2}\,$, and consider the following positive harmonic functions in $(0,L):$ 
	$$v_{1,0}(t):=\sqrt{2}\,t, \qquad v_{2,0}(t):=\dfrac{L-t}{\sqrt{2}L}.$$
	 For all $k\geq 0$, let 
	$$
	v_{1,k+1}=\sqrt{c_1v_{1,k}^2+c_2v_{2,k}^2+2c_3v_{1,k}v_{2,k}}\, ,\qquad
		v_{2,k+1}=	v_{1,k+1}\int\limits_t^L\dfrac{1}{v_{1,k+1}^2}\ds.
	$$
Then
\begin{align*}
G_{k+1}(t)\!\!&:=\!\!\!
\int_t^L\dfrac{1}{v_{1,k+1}^2}\ds=\!\!
\int_t^L
\dfrac{1}{v_{1,k}^2(s)}
\dfrac
{1}{c_1+c_2 \left (\displaystyle{\int_s^L}\dfrac{1}{v_{1,k}^2}\dz \right )^{\!\!2}+2c_3\displaystyle{\int_s^L}\dfrac{1}{v_{1,k}^2}\dz}\ds
\\ &
=
\int\limits_{0}^{G_k(t)} \dfrac{1}{c_1+c_2\tau^2+2c_3 \tau} \text{d}\tau=F(G_k)-F(0),
\end{align*}
where $F(\tau)=\displaystyle{\int} \dfrac{1}{c_1+c_2\tau^2+2c_3 \tau} \text{d}\tau$ and $G_0(t):=(L-t)/(2Lt)$. So, 
$$
G_{k}(t)=
\begin{cases}
F\left(G_{k-1}(t)\right)-F(0)   &  k\geq 1 ,\\[3mm]
\dfrac{L-t}{2Lt} &  k=0. 
\end{cases}
$$
By Proposition \ref{2.17}  the function 
$$
w(t):=\sum\limits_{k=1}^{\infty} \dfrac{1}{v_{1,k}^4}=
\sum\limits_{k=1}^{\infty} (G_{k}')^2 
$$
is a Hardy-weight for $Ly=-y''$ in $(0,L).$ Furthermore, if $0<c_1\leq 1/L$, then $$w>w_{\mathrm{class}}= (2t)^{-2}.$$ Moreover, if the assumptions of Lemma \ref{lemma 5.5} are satisfied for a subcritical  operator $P$ in $\Gw$, then  $w(G_{\varphi}/u)|\nabla (G_{\varphi}/u)|_A^2$ is a Hardy-weight for $P$ in $\Gw.$
We remark that the improved Hardy inequality obtained in \cite{FT} is recovered  
with $c_1=1/L$, $c_2=0$, and $c_3=1$. 
}
\end{example}

	\section*{Appendix}\label{eq:app}
In this short section we explain how to obtain the generalized eigenfunctions $u_\xi$ given in Proposition~\ref{SL_properties}. 	
Let $\rho,q\in L^1_{\text{loc}}(\alpha,\beta),~\lambda \in \R$ with $\rho>0$ in $(\ga,\gb)$, and consider the equation: 
$$
-y''(t)+q(t)y(t)=\lambda \rho(t) y(t) \qquad t\in(\alpha,\beta).
$$
The Liouville's substitution 
$$
y(t)=\frac{w(s)}{\sqrt[4]{\rho}}\,, \quad s(t)=\int_{\alpha}^{t} \sqrt{\rho(\zeta)}\text{d}\zeta, 
$$
yields Liouville's normal form (see for example \cite[Chapter 10.9]{BR})
$$
-w''(s)+\hat{q}(s)w(s)=\lambda w(s) \qquad s\in (0,s(\beta)),
$$
where 
$$
\hat{q}=-1+\frac{1}{4\rho} \left [ \left (\frac{\rho '}{\rho} \right )'-\frac{1}{4} \left ( \frac{\rho'}{\rho}\right )^2\right ].
$$
Note that $\hat{q}=0$ if and only if the function $v:=\rho^{-1/4}$ satisfies the Ermakov-Pinney equation 
$$-v''=\frac{1}{v^3}\, .$$ Moreover,  $v$ is a  positive solution of   the equation
$-y''-\dfrac{y}{v^4}=0$ having minimal growth near $0$ if and only if $v(0)=0$, namely, $v(t)=\sqrt{2t-at^2}$.

Now consider the equation 
$$
-y''-\frac{ y}{(2t-at^2)^2}=\xi^2\frac{ y}{(2t-at^2)^2} \qquad 
t\in (2/(a+M\text{e}^{\pi/\xi}),2/(M+a)).
$$
where $a,M,\xi$ are given in Proposition \ref{SL_properties}. 
The Liouville's substitution 
$$
y(t)=w\sqrt{2t-at^2}, \quad s=\int_{2/(a+M\text{e}^{\pi/\xi})}^{t} (2\zeta-a\zeta^2)^{-1}\text{d}\zeta 
=\frac{1}{2}\log \left (\frac{Mt}{2-at} \right )+
\frac{\pi}{2 \xi}
$$
implies $-w''=\xi^2 w$, and therefore,
$w=C_1\cos(\xi s)+C_2\sin(\xi s)$. The boundary conditions given by
Proposition \ref{SL_properties} then imply
$$
u=C_1\sqrt{2t-at^2}\sin\left(\frac{\xi}{2}\log \left (\frac{Mt}{2-at} \right )+\frac{\pi}{2} \right )
=C_1\sqrt{2t-at^2}\cos\left(\frac{\xi}{2}\log \left (\frac{Mt}{2-at} \right ) \right ).
$$

\medskip

\begin{center}
{\bf Acknowledgments}
\end{center}
The authors wish to express their sincere gratitude to Martin Fraas who initiated the main problem studied in this paper.  The  authors  acknowledge  the  support  of  the  Israel  Science Foundation (grants No.  970/15 and 637/19) founded by the Israel Academy of Sciences and Humanities. 
\bibliographystyle{apa}

\begin{thebibliography}{9}
			\bibitem{AG}
			S.~Agmon, ``Lectures on Exponential Decay of Solutions of Second-Order Elliptic Equations: Bounds
		on Eigenfunctions of $N$-body Schr\"odinger Operators", Math. Notes, vol. 29, Princeton University Press, Princeton, 1982.
		
		\bibitem{Agmon}  S.~Agmon, On positivity and decay of solutions of second order
		elliptic equations on Riemannian manifolds, {\em in} ``Methods of
		Functional Analysis and Theory of Elliptic Equations" (Naples,
		1982), pp. 19--52, Liguori, Naples, 1983.
			
			\bibitem{Ancona01} A.~Ancona, Negatively curved manifolds, elliptic operators, and the Martin boundary, {\em Ann. of Math.} {\bf 125} (1987),  495--536.
	
			\bibitem{Ancona02} A.~Ancona, Some results and examples about the behavior of harmonic functions and Green's functions with respect to second order elliptic operators, {\em Nagoya Math. J.} {\bf 165} (2002), 123--158.
			
			\bibitem{BEL} A.~A.~Balinsky, W.~D.~Evans, and R.~T.~Lewis: ``The Analysis and Geometry of Hardy's Inequality", Universitext. Springer, Cham, 2015.
				
					\bibitem{Barbatis} G.~ Barbatis, S.~ Filippas, and A.~Tertikas, 
				Series expansion for $L^p$ Hardy inequalities, \textit{Indiana Univ. Math. J.} \textbf{52} (2003), 171--190. 
				
					\bibitem{BFT}	G.~Barbatis, S.~Filippas, and A.~Tertikas, 
				Sharp Hardy and Hardy-Sobolev inequalities with point singularities on the boundary,\textit{ J. Math. Pures Appl.}  \textbf{117} (2018), 146--184.  
				\bibitem{BR}
				 G.~Birkhoff, and G.~Rota, "Ordinary Differential Equations", Fourth Edition, John Wiley and Sons, Inc., New York, 1989. 
				\bibitem{BM} H.~Brezis, and M.~Marcus, Hardy's inequalities revisited, {\em Ann. Scuola Norm. Sup. Pisa Cl. Sci.} (4) {\bf 25} (1997),  217--237. 
				
				\bibitem{CFKS} H.~L.~Cycon, R.~G.~Froese, W.~Kirsch, and B.~Simon, ``Schr\"odinger Operators with Application to Quantum Mechanics and Global Geometry", Texts and Monographs in Physics. Springer Study Edition. Springer-Verlag, Berlin, 1987.
								
				\bibitem{D} E.~B.~Davies, A review of Hardy inequalities, {\em  in:} ``The Maz'ya Anniversary Collection" {\bf 2}, 55--67, Oper. Theory Adv. Appl., 110, Birkh\"auser, Basel, 1999.
				
					\bibitem{DFP} B.~Devyver, M.~Fraas, and Y.~Pinchover, Optimal Hardy weight for second-order elliptic operator: an answer to a problem of Agmon, \textit{J. Funct. Anal.} \textbf{266} (2014), 4422--4489.

					\bibitem{DP1}
				B.~Devyver, and Y.~Pinchover, Optimal $L^p$ Hardy-type inequalities,  \textit{Ann.  Inst.  H. Poincar\'{e}. Anal. Non Lin\'{e}aire} \textbf{33} (2016), 93--118.
	
\bibitem{E}  V.~P.~Ermakov, Second order differential equations: conditions of complete integrability, {\em Universita Izvestia Kiev}, Series {\bf III 9} 1--25 (Russian), Translated from the 1880 Russian original by A.~O.~Harin and edited by P.~G.~L.~Leach, {\em Appl. Anal. Discrete Math.} {\bf 2} (2008), 123--145.

	\bibitem{FT}
	S.~Filippas, and A.~Tertikas, Optimizing improved Hardy inequalities, {\em J. Funct. Anal.} {\bf 192} (2002), 186--233.
 %
 \bibitem{H} F.~Haas, The damped Pinney equation and its applications to dissipative quantum mechanics, {\em Phys. Scr.}  {\bf 81} (2010),  025004 (7 pp.).
%
 \bibitem{KPP} M.~Keller, Y.~Pinchover, and F.~Pogorzelski, Optimal Hardy inequalities for Schr\"odinger operators on graphs, {\em Comm. Math. Phys.} {\bf 358} (2018), 767--790. 
 %
 \bibitem{KPP1} M.~Keller, Y.~Pinchover, and F.~Pogorzelski, Criticality theory for Schrödinger operators on graphs, to appear in {\em J. Spectr. Theory}, 34 pp., arXiv: 1708.09664.
%
\bibitem{KP}	H.~Kova\v{r}\'{\i}k, and Y.~Pinchover, On minimal decay at infinity of Hardy-weights, to appear in \textit{Commun. Contemp. Math}, 16 pp., arXiv: 1812.01849.  
%
	\bibitem{KML} A. ~Kufner, L.~Maligranda, and L.~Persson, ``The Hardy Inequality: About its History and Some Related Results", Vydavatels\'y Servis, Plzen, 2007. 
	
	\bibitem{KO}  A.~Kufner, and B.~Opic, ``Hardy-type Inequalities", Pitman Research Notes in Math., Vol.	219, Longman, Harlow,1990.
	%
	\bibitem{LL}  E.~Lieb, and M.~Loss,  ``Analysis", Graduate Studies in Mathematics 14. American Mathematical Society, Providence, RI, 2001.
	%
		\bibitem{MU} M.~Murata, Structure of positive solutions to $(-\Delta+V)u=0$ in $\R^n$, \textit{Duke Math. J.} {\bf53} (1986), 869--943.
	%
	\bibitem{P3} Y.~Pinchover, 	Topics in the theory of positive solutions of second-order elliptic and parabolic partial differential equations.  in: ``Spectral Theory and Mathematical Physics: a Festschrift in Honor of Barry Simon's 60th Birthday",  eds. F. Gesztesy, et al., Proceedings of Symposia in Pure Mathematics 76 Part 1, American Mathematical Society, Providence, RI, 2007, 329--356.
	%
		\bibitem{PP} Y.~Pinchover, and G.~Psaradakis, On positive solutions of the $(p,A)$-Laplacian with potential in Morrey space, \textit{Anal. PDE} \textbf{9} (2016), 1317--1358. 
	%
	\bibitem{PT} Y.~Pinchover, and K.~Tintarev, Ground state alternative for p-Laplacian with potential term, \textit{Calc. Var. Partial Differential Equations} \textbf{28} (2007), 179--201.
%
\bibitem{PI} E.~Pinney, The nonlinear differential equation $y''+ p(x) y+ cy^{-3}= 0$,\textit{ Proceedings of the American Mathematical Society} {\bf 1} (1950), 681.
		%
\bibitem{Z1}	A.~Zettl, ``Sturm-Liouville Theory",  Mathematical Surveys and Monographs 121. American Mathematical Society, Providence, RI, 2005.
	\end{thebibliography}

\end{document}